\documentclass[11pt]{amsart}

\usepackage{amsmath,amsthm,amssymb}
\usepackage{amsfonts}
\usepackage{mathrsfs}
\usepackage[square,sort,comma,numbers]{natbib}

\usepackage{graphicx}
\usepackage[labelfont={bf}]{caption} 
\usepackage{subcaption}
\usepackage[usenames,dvipsnames]{xcolor}
\usepackage{float} 
\usepackage{setspace} 
\usepackage{fancyhdr}

\newtheorem{thm}{Theorem}[section]
\newtheorem{prop}[thm]{Proposition}
\newtheorem{lem}[thm]{Lemma}

\newtheorem{cor}[thm]{Corollary}

\newtheorem{lemma}[thm]{Lemma}

\newtheorem{defn}[thm]{Definition}
\let\olddefn\defn
\renewcommand{\defn}{\olddefn\normalfont}

\theoremstyle{definition}

\newcommand{\ibd}{\partial_{\infty}}
\newcommand{\ra}{\rightarrow}

\newcommand{\ep}{\epsilon}
\newcommand{\head}{\noindent \textbf}

\newcommand{\Z}{\mathbb{Z}}

\newcommand{\E}{\mathbb{E}}

\newcommand{\T}{\mathscr{T}}

\newcommand{\tbd}{\partial_{T}}
\newcommand{\bd}{\partial}

\makeatletter 
\@addtoreset{thm}{Chapter} 
\makeatother

\usepackage{tikz}
\usetikzlibrary{calc}
\usetikzlibrary{arrows}
\usetikzlibrary{decorations.pathreplacing}
\usetikzlibrary{intersections}

\begin{document}
\title{Actions of Right-angled Coxeter groups on the Croke Kleiner Spaces}
\author{Yulan Qing}

\address{Department of Mathematics \\ Technion - Israel Institute of Technology \\ Haifa, 32000 \\ Israel}
\email{yulan.qing@gmail.com}


\date{\today}

\subjclass[2000]{20F65}

\keywords{CAT(0) space, Right-angled Coxeter group, visual boundary}

\begin{abstract}
It is an open question whether right-angled Coxeter groups have unique group-equivariant visual boundaries. In \cite{ckpaper}, Croke and Kleiner present a right-angled Artin group with more than one visual boundary. In this paper we present a right-angled Coxeter group with non-unique equivariant visual boundary. The main theorem is that if right-angled Coxeter groups act geometrically on a Croke-Kleiner spaces constructed in \cite{ckpaper}, then the local angles in those spaces all have to be $\pi / 2$. We present a specific  right-angled Coxeter group with non-unique equivariant visual boundary.  However, we conjecture that the right angled Coxeter groups that can act geometrically on a given CAT(0) space are far from unique.
\end{abstract}

\maketitle

\section{Introduction}
The question that motivated this study is the uniqueness of visual boundaries for right-angled Coxeter groups. The open questions addressed are the following:
\begin{enumerate}
 \item Are all CAT(0) visual boundaries of a given right-angled Coxeter group (abstractly) homeomorphic? \\
\item  Are all CAT(0) visual boundaries of a given right-angled Coxeter group equivariantly homeomorphic? 
\end{enumerate}
In this paper we answer the second question negatively.  We show that the first question is related to the larger question of whether a certain length variation can change the homeomorphism type of the visual boundary of a  CAT(0) cube complex. In his thesis \cite{mob}, O'Brien carefully studied the existence and uniqueness of strict fundamental domains of right-angled Coxeter groups acting on CAT(0) spaces. A key lemma in this paper is largely an application of his main result.\\

To investigate the question of unique boundary, we consider the action of a right-angled Coxeter group on a family of spaces constructed by Croke-Kleiner. The family is constructed from a fixed CAT(0) cube complex $X$ by varying certain geometric data while preserving the CAT(0) property.  We assume that a right-angled Coxeter group acts geometrically on the spaces in the family.   We want to know whether varying some or all of the geometric data $(\theta_1, \theta_2, \theta_3, a, b, c, d)$ changes the equivariant homeomorphism type of its boundary. The main results imply two things. First, there are only certain geometric data changes that are allowed if we insist on having a right-angled Coxeter group acting on one of these spaces.  Second, these allowed changes do change the equivariant homeomorphism type of the boundary and we can find a right-angled Coxeter group $W$ that acts geometrically on this space thus answering the question about the uniqueness of equivariant visual boundaries.\\

The key step is the first result.  We prove that unlike in the case of right-angled Artin groups, the Croke-Kleiner space does not support a geometric action by a right-angled Coxeter group if the three angles on the tori $\theta_1, \theta_2, \theta_3$ are not fixed at $\pi/2$. \\

Take $X$ as defined in Section 3. There are many embedded flats in $X$. Among them are the flats that are universal covers of the tori $T_i$, which we call \textit{special flats}.
In this paper we prove that:

\begin{thm}
 Suppose $W$ is a right-angled Coxeter group acting geometrically on the Croke-Kleiner complex while preserving the family of special flats.  Then the angles $\theta_1, \theta_2, \theta_3$ must 
 all be right angles.
\end{thm}

This implies that the "right-angled" in the terminology "right-angled Coxeter group" turns out to be literal and is consistent with the "geometric" property of the group. Furthermore, given the result from \cite{flow} that if we fix the gluing angle of the Croke-Kleiner space at $\pi$/2 and change the side lengths of the tori, the resulting boundaries are not equivariantly homeomorphic to each other, we conclude the following:

\begin{cor}
There exists a right-angled Coxeter group that does not have unique equivariant visual boundary.
\end{cor}

Aside from answering the open question, this close examination of the interplay between right-angled Coxeter actions and the CAT(0) geometry of the Croke-Kleiner spaces shows that right-angled Coxeter groups can be more "geometrically rigid" than their counterparts in the class of right-angled Artin groups.\\

The more general open questions we aim to contribute to is the following:  
\begin{enumerate}
\item Suppose $X$ is a CAT(0) square complex.  If we  change the cubes to rectangles, does the homeomorphism type of the boundary change? 
\item  If one has a group acting geometrically on $X$, then does the equivariant homeomorphism type of the boundary change when we change cubes to rectangles?  \\
\end{enumerate}
\section{Preliminaries}
In this section we give basic definitions and facts concerning CAT(0) spaces, boundaries and quasi-isometries all of whose proofs can be found in \cite{thebible}.   We also give the definitions and facts we need concerning right-angled Artin and Coxeter groups.

\subsection{CAT(0) Spaces and their boundaries}


%

A metric space $X$ is CAT(0) if geodesic triangles in $X$ are at least as thin as a triangle in Euclidean space with the same side lengths.  
It follows immediately from the definition that CAT(0) spaces are uniquely geodesic and thus contractible via geodesic retraction to a base point in the space.\\

Recall that a metric space $X$ is {\it proper} if closed metric balls are compact.  In this case, $X$ can be compactified via the \textit{visual boundary} of $X$.  The points of this boundary are equivalence classes of geodesic rays defined as follows:\\

A \textit{geodesic ray} in $X$ is a geodesic $c: [0, \infty) \ra X$. Consider the set of geodesic rays in $X$. Two geodesic rays $c_1$ and $c_2$ are said to be \textit{asymptotic} if $f(t): = d(c_1(t), c_2(t))$ is a bounded function.  The set of equivalence classes is denoted by $\bd X$ and called the \textit{visual boundary} of $X$. If $\xi \in \bd X$ and $c$ is a geodesic ray belonging to $\xi$, we write $c(\infty) = \xi$.\\

The following is a basic lemma in CAT(0) geometry:

\begin{lemma} For any $\xi \in \bd X$ and any $x \in X$, there is a unique geodesic ray $c_{x \xi}: [0, \infty) \ra X$ with $c_{x \xi}(0) = x$ and $c_{x \xi}(\infty) = \xi$. The image of $c_{x \xi}$ is denoted by $x \xi$.\\
\end{lemma}

There are two topologies we can put on $\bd X$.
Set $\overline{X} = X \bigcup \bd X$. The \textit{cone topology} on $\overline{X}$ has as a basis the open sets of $X$ together with the sets
\begin{center}
$U(x, \xi, R, \ep) = \{ z \in \overline{X} | x \notin B(x, R), d(c_{xz}(R), c_{x\xi}(R)) < \ep \}$
\end{center}
where $x \in X$, $\xi \in \bd X$ and $R > 0, \ep >0$. The topology on $X$ induced by the cone topology coincides with the metric topology on $X$.\\

This topology looks as if it depends on the base-point $x$ in the above description of open sets, however the previous lemma shows that there is a natural change of base-point homeomorphism when the base-point is changed.  \\

\begin{figure}[h]
\begin{center}
\begin{tikzpicture}[scale=0.6]

\node (x) [circle,fill,inner sep=1pt,label=180:$x_0$] at (0,0) {};
\draw [name path=circle] (0,0) circle (3);

\draw [name path=line,thin] (0,0) to [bend right=20] (5,2.5);
\draw [thin] (0,0) to [bend left=20] (5,-2.5);
\draw (0,0) to (5,0) node [right] {$\xi$};

\draw [very thin,
	decorate,
	decoration={brace,mirror,amplitude=20pt}] 
	(0,0) -- (3,0) node [midway,yshift=-30pt] {$r$};

\draw [very thin,
	name intersections={of=line and circle},
	decorate,
	decoration={brace,amplitude=7pt}] 
	(intersection-1) -- (3,0) node [midway,xshift=12pt] {$\epsilon$};

\end{tikzpicture}
\end{center}
\caption{A basis for open sets}
\label{}
\end{figure}
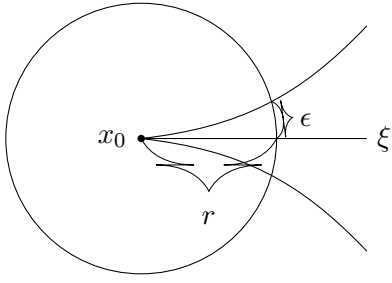

The set $\bd X$ together with the cone topology is called the \textit{visual boundary} of $X$, denoted $\ibd X$.\\

There is another way to put a topology on $\bd X$ using the Tits metric.  Let $c_1, c_2: [0, \infty) \ra X$ be two geodesic rays with $c_1(0) = c_2(0) = x$ which belongs to two equivalent classes  $\xi, \eta$. Let $\angle_x (\xi, \eta)$ denote the \textit{Alexandrov angle} between the rays $\xi$ and $\eta$. For $t_1, t_2 \in (0, \infty)$, $\angle_x(\xi, \eta)$ is defined to be
$$
\angle_x(\xi, \eta) = \lim_{t_1, t_2 \rightarrow 0} \overline{\angle}_{x} (c_1(t_1), c_2(t_2)
$$

Where $\overline{\angle}_{x} (c_1(t_1), c_2(t_2)$ denotes the angle at the vertex $x$ in a comparison triangle $\overline{\Delta}(x, c_1(t_1), c_2(t_2)) \subset \E^2$.  
 
If $X$ is a complete $CAT(0)$ space, then the angle $\angle(\xi, \eta)$ between $\xi, \eta$ is defined to be
$$
\angle(\xi, \eta) = \sup_{x \in X} \angle_x (\xi, \eta)
$$
It can be shown that the angle function is a metric on the equivalence classes of all geodesic rays and we define the \textit{Tits metric} to be the length metric associated to the angular metric, denoted $d_T$.  The set $\bd X$ together with the topology arising from $d_T$ is called the \textit{Tits boundary} of $X$, denoted $\tbd X$.\\

In general the visual boundary of the space is different from the Tits boundary of the space. For example, 
consider $X=\mathbb{H}^2$.  For any two points $\xi, \eta$ on the boundary, there is a geodesic line in $\mathbb{H}^2$ joining $\xi$ and $\eta$. Therefore $\angle (\xi, \eta) = \pi$ for any two distinct points one the boundary, hence $\tbd \mathbb{H}^2$ is a discrete set yet $\partial_{\infty}\mathbb H^2$ is homeomorphic to $S^1$ with the cone topology.\\ 

\subsection{Quasi-Isometry and Quasi-Isometric Embeddings}



\begin{defn}
Let $(X_1 , d_1 )$ and $(X_2 , d_2 )$ be metric spaces. A
(not necessarily continuous) map $f : X_1 \ra X_2$ is called a $(\lambda, \epsilon)$-\textit{quasi-isometric
embedding} if there exist constants $\lambda \geq 1$ and $\epsilon \geq 0$ such that for all $x, y \in X_1$
$$
\frac{1}{\lambda} d_1 (x, y) - \epsilon  \leq d_2 (f (x), f (y)) \leq \lambda d_1 (x, y) + \epsilon
$$

If, in addition, there exists a constant $C \geq 0$ such that every point of $X_2$ lies in the
$C$-neighborhood of the image of $f$ , then $f$ is called a $(\lambda, \ep)$-quasi-isometry.
When such a map exists, $X_1$ and $X_2$ are said to be \textit{quasi-isometric}.\\
\end{defn}

The following is often referred to as the Fundamental Theorem of Geometric Group Theory:

\begin{thm}($\check{S}$varc-Milnor lemma)
If $X$ is a complete, locally compact geodesic metric space and $G$ acts geometrically on $X$, 
then $G$ is finitely generated and X is quasi-isometric to every Cayley graph for
$G$.
\end{thm}

\subsection{Right-angled Groups}

\begin{defn}
An \textit{Artin group} A is a group with presentation of the form
$$
A = \langle s_1, s_2 . . , s_n | ( s_i s_j )^{m_{ij}} = (s_j s_i)^{m_{ji}}
\text{ for all } i \neq j \rangle
$$
where $m_{ij} = m_{ji} \in \{2,3,...\infty \}$. $(s_i s_j)^{m_{ij}}$ denotes an alternating product of $s_i$ and $s_j$ of length $m_{ij}$, beginning with $s_i$. If $m_{ij} = \infty$ , then there is (by convention) no relation for $s_i$ and $s_i$.

 A \textbf{right-angled Artin group} 
\cite{raag} is one in which $m_{ij} \in \{2, \infty \}$ for all
$i, j$. In other words, in the presentation for the Artin group, all relations are commutator
relations: $$s_i s_j  = s_j s_i$$

\end{defn}

The easiest way to specify a presentation for a right-angled Artin group
is by means of a defining graph. This is a graph whose vertices are labeled by the
generators $S = \{ s_1, . . . , s_n \}$ and whose edges connect pairs of vertices 
$s_i, s_j$ if and only if
$m_{ij} = 2$. Note that any finite, simplicial graph $\Gamma$ is the defining graph for a right-angled Artin group.  

\begin{defn}
Formally, a \textit{Coxeter group} can be defined as a group with a presentation of the following form:
$$
\langle s_1, s_2,...s_n | (s_i)^2 = 1, (s_i s_j)^{m_{ij}} = 1, \text{ where } m_{ij} \in \{2, 3, 4,...\infty \}  \rangle
$$

A \textbf{right-angled Coxeter Group} \cite{localcon2} is where $m_{ij} \in \{2, \infty \}$

\end{defn}

Just as for right-angled Artin groups, the presentation for a right-angled Coxeter group can be given by a finite simplicial graph with the understanding that each vertex now represents a generator of order 2.  \\

Here are more basic facts of right-angled Coxeter Groups \cite{moracg}:

\begin{itemize}
\item  If $s_i$ is not adjacent to $s_j$, then the order of $s_i s_j$ is 
infinite.
\item A right-angled Coxeter group is abelian if and only if it is
finite which is true if and only if the defining graph is complete.  
\item If $w$ has
finite order, then $w^2 = 1$. Right-angled Coxeter groups are distinguished from other Coxeter groups by this fact. That is, if
every finite order element of a Coxeter group is two, then the Coxeter
group is right-angled.
\item A right-angled Coxeter group $W$ has a non-trivial center if and only if it can
be written as $W'′ \times \Z_2$ for a right-angled Coxeter group $W'$.

\end{itemize}

\subsection{Strict Fundamental Domain}

In this section, we suppose that $G$ is a group acting on a metric space $X$ by isometries. \\

The group of all isometries from a metric space $(X, d)$ to itself will be denoted $\textrm{Isom}(X)$. If 
$G \subset \textrm{Isom}(X)$, then we say that $G$ acts on $X$ by isometries.

\begin{defn}
A group $G$ acts \textit{geometrically} on a metric space $X$ if $G$ acts properly discontinuously, cocompactly, and by isometries.
\end{defn}

\begin{defn}  Suppose $G$ is a group acting on a metric space $X$. The \textit{fixed point set} of $S \subset G$ on the space $X$ is the set
$$
\text{Fix}(S):=\{ x \in X | g.x = x \text{ for all } g \in S \}
$$ 
\end{defn}

For each point $x \in X$, the \textit{orbit} of $x$ is the set
$$
\mathcal O_x:=\{ y \in X | y = g·x \text{ for some } g \in G\}
$$

\begin{defn}
Let $K$ be a closed subset of $X$.  $K$ is a
\textit{strict fundamental domain} of $G$ on $X$ if every orbit meets $K$ exactly once.
\end{defn}

According to Bass-Serre Theory, in the case of a group acting on a tree, a strict fundamental domain determines an amalgamated product decomposition of the group.  Conversely, if the original group has an amalgamated product decomposition, then there is a tree (unique up to isomorphism) on which $G$ acts, with an edge as the strict fundamental domain. The critical tools to our study come from the thesis of O'Brien.\\

\begin{prop}\cite{mob} 

Suppose $G$ acts geometrically on a uniquely geodesic space $X$ with a
strict fundamental domain $K$ whose translates are locally finite. Then

\begin{enumerate}
\item $K$ is convex, and
\item the the quotient $X / G$ is isometric to $K$
\end{enumerate}
\end{prop}

The full strength of O'Brien's result generalizes the Bass-Serre theory from trees to other spaces on which a strict fundamental domain can be found.  By O'Brien\cite{mob}, the action of a Coxeter group $W$ on a space $X$ has a strict fundamental domain if and only if for every $x \in X$ and $w \in W$, every path from $x$ to $w.x$ meets the fixed point set of $X$.  When the action is on a tree, this condition(which is called \textit{generalized reflection}) is satisfied.  Once we obtain a strict fundamental domain $K$, we can study the stabilizer groups of the topological boundary of $K$ and recover an amalgamated product decomposition of the group. 

Specific to our result, we let a right-angled Coxeter group act on the nerve tree of a Croke-Kleiner space (terms to be defined in the next section) and use the following theorem:

\begin{thm}\cite{mob}
If a $G$ acts geometrically on a tree, then there is a strict fundamental domain that is a finite sub-tree.
\end{thm}

\section{Croke-Kleiner Spaces}

It is well-known that if $A$ a right-angled Artin group, then $A$ acts geometrically on a CAT(0) cube complex $X(A)$.  The spaces constructed here come from the cube complex for a specific right-angled Artin group whose defining graph is:

\begin{figure}[h]
\begin{center}
\begin{tikzpicture}[scale=1.0]

\node (v1) [circle,fill,inner sep=2pt] at (0,0) {};
\node (v2) [circle,fill,inner sep=2pt] at (1,0) {};
\node (v3) [circle,fill,inner sep=2pt] at (2,0) {};
\node (v4) [circle,fill,inner sep=2pt] at (3,0) {};

\draw (v1) -- (v2) -- (v3) -- (v4);

\end{tikzpicture}
\end{center}
\caption{}
\label{}
\end{figure}

The cubical space $X(A)$ used by Croke and Kleiner in \cite{ckpaper} is the universal cover
of a torus complex $Y$. Start with a flat torus $T_2$ with the property that a pair $b$, $c$
of unoriented, $\pi_{1}$-generating simple closed curves in $T_2$ meets at a single point at an angle $\alpha= \frac{\pi}{2}$. Let $b, c$ have length 1. Let $T_1, T_3$ be flat tori containing simple closed essential loops, $a, b_1$ and $c_1, d$,  respectively, such that $length(b_1) = length(b)$, $length(c_1) = length(c)$. Let $a, d$ also have length 1. Let $Y$ be the union of $T_1, T_2, T_3$ with $b_1$ identified isometrically with $b$ and $c_1$ with $c$. Let $X$ be the universal cover
of $Y$. Let $Y_1 = T_1 \cup T_2 $, and let $Y_2 = T_3 \cup T_2 $, $X_i$ be the universal cover of $Y_i$ in $X$.  That $X$ is CAT(0) cubical follows from the Equivariant Gluing theorem \cite{thebible}.\\

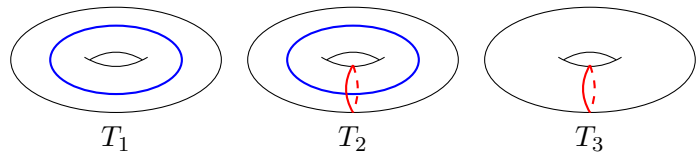
\begin{figure}[h]
\begin{center}
\begin{tikzpicture}[scale=0.35]

\draw (-1,0) to [bend left] (1,0);
\draw (-1.2,0.1) to [bend right] (1.2,0.1);
\draw (0,0) ellipse (4 and 2);
\draw [thick,blue] (0,0) ellipse (2.5 and 1.3);

\draw (8,0) to [bend left] (10,0);
\draw (7.8,0.1) to [bend right] (10.2,0.1);
\draw (9,0) ellipse (4 and 2);
\draw [thick,blue] (9,0) ellipse (2.5 and 1.3);
\draw [thick,red,bend right=30] (9,-0.2) to (9,-2);
\draw [thick,red,dashed,bend left=20] (9,-0.2) to (9,-2);

\draw (17,0) to [bend left] (19,0);
\draw (16.8,0.1) to [bend right] (19.2,0.1);
\draw (18,0) ellipse (4 and 2);
\draw [thick,red,bend right=30] (18,-0.2) to (18,-2);
\draw [thick,red,dashed,bend left=20] (18,-0.2) to (18,-2);

\node at (0,-3) {$T_1$};
\node at (9,-3) {$T_2$};
\node at (18,-3) {$T_3$};

\end{tikzpicture}
\end{center}
\caption{Tori Complex}
\label{}
\end{figure}

One obtains an uncountably infinite family of CAT$(0)$ spaces by changing the geometry of $X(A)$ in such a way that it is no longer cubical yet it is still CAT(0) and the group $A$ still acts on the new spaces geometrically.   Specifically, one can change the angle $\alpha$ to be any real number $0 < \alpha\leq \frac{\pi}{2}$.  This particular change of geometry was studied in the original paper \cite{ckpaper} where they proved that changing the angle from $\frac{\pi}{2}$ to any other value changes the homeomorphism type of $\partial_{\infty} X(A)$.  This was further investigated by J. Wilson in \cite{wilson} where it was shown that any two different angles give different visual boundaries.

The \textit{geometric data} associated with a Croke-Kleiner space consists of three intersecting angles and four translation distances. The three intersection angles are that of the intersecting angle of the three pairs of $\pi_1$-generating, simple closed curves on the three tori, which we denote $\theta_1, \theta_2,\theta_3$. The four lengths are the translation distance of $a, b, c, d$.  

In this paper we fix $\theta_1 = \theta_2 = \theta_3 = \pi/2 $ unless otherwise specified. Since the $\theta_i$ are right angles, we can use $|a|, |b|, |c|, |d|$ to denote the translation distance, or by abuse of notation $a, b, c, d$. It can be easily checked that length variation is a quasi-isometry but not an isometry on the space. This motivates the following definition. 

\begin{defn} A CAT$(0)$ $X$ space obtained by a length change on the cubical complex $X(A)$ will be referred to as a \textit{Croke-Kleiner space}.
\end{defn}

We now describe the structure of any Croke-Kleiner space $X$.  

\begin{defn}
A \textit{barrier} is a maximal connected component of the universal cover of $T_2$ in $X$. 
A \textit{block} is a maximal, connected component of the universal cover of $Y_i$ in $X$, which we denoted $X_i$. 
\end{defn}

Each block, as well as each barrier is a closed, connected and locally convex subset of $X(A)$. Let $\mathscr{B}$ and $\mathscr{W}$ denote respectively the collection of all blocks and barriers. We will prove later that $\mathscr{B}$ and $\mathscr{W}$ are countably infinite sets.\\

Let $\T_4$ be the regular 4-valence, infinite tree that is isomorphic as a graph to the Cayley graph of $F_2 $ with two generators. A block is isometric to the metric product of a $\T_4$ with appropriate edge lengths with the real line $\mathbb{R}$. The intersection of two blocks can be either an empty set or a barrier. Two blocks are \textit{adjacent} if and only if their intersection is a barrier.

\begin{figure}[h]
\begin{center}
\begin{tikzpicture}[scale=0.85]

\begin{scope}[z={(-7mm,-4mm)}]

\draw (-2,2,0) -- (2,2,0) -- (2,-2,0) -- (-2,-2,0) -- (-2,2,0);
\draw (-2,0,2) -- (2,0,2) -- (2,0,-2) -- (-2,0,-2) -- (-2,0,2);
\draw (-2,0.5,-1) -- (2,0.5,-1) -- (2,-0.5,-1) -- (-2,-0.5,-1) -- (-2,0.5,-1);
\draw (-2,0.5,1) -- (2,0.5,1) -- (2,-0.5,1) -- (-2,-0.5,1) -- (-2,0.5,1);
\draw (-2,1,0.5) -- (2,1,0.5) -- (2,1,-0.5) -- (-2,1,-0.5) -- (-2,1,0.5);
\draw (-2,-1,0.5) -- (2,-1,0.5) -- (2,-1,-0.5) -- (-2,-1,-0.5) -- (-2,-1,0.5);

\draw [dashed] (-2,0,0) -- (2,0,0);
\draw [dashed] (-2,1,0) -- (2,1,0);
\draw [dashed] (-2,-1,0) -- (2,-1,0);
\draw [dashed] (-2,0,1) -- (2,0,1);
\draw [dashed] (-2,0,-1) -- (2,0,-1);

\end{scope}

\end{tikzpicture}
\end{center}
\caption{A Block}
\label{}
\end{figure}

The Croke-Kleiner space can be projected onto an infinite, locally infinite tree via the following: let each block be projected onto a vertex; two vertices are adjacent if and only if two blocks are adjacent. If a group acts on the Croke-Kleiner space, it necessarily preserves blocks and block adjacencies. Thus the group acts on the tree.

\subsection{The boundaries of $X$}
Let $X$ be a Croke-Kleiner space and let $\ibd X$ and $\tbd X$ denote respectively the visual boundary and the Tits boundary of the space $X$. $\ibd B$ is homeomorphic to the suspension of a Cantor set and $\tbd B$ is the suspension of an uncountable discrete set with each suspension arc having length $\pi/2$.

\begin{defn}
The two suspension points of $\ibd B$, called \textit{poles} of the block, are the equivalence classes of geodesics correspond with the pair $\{b^{n} (x)$, $b^{-n}(x)\}$, or the pair $\{c^{n} (x)$, $c^{-n}(x)\}$, as $n \rightarrow \infty$.
\end{defn}

\begin{defn}
A \textit{longitude} of the block is an arc in $\ibd B$ joining the two poles. It can also be thought of as the suspension of a point in the Cantor set.
\end{defn}

We say that a geodesic ray $\xi$ \textit{enters} a special flat $V$ if there are values $r < R$ in the domain
of $\xi$ such that $\xi([r, R]) \subset V$. $\xi$ enters a block if it enters a non-barrier special flat 
of the block. 

\begin{defn}
An \textit{itinerary} of a geodesic ray, $It(\xi)$, is the sequence of blocks that the 
geodesic ray enters in order. An itinerary can be either finite or infinite. 
\end{defn}

We say that $\xi \in \ibd X$ is a \textit{vertex} if there is a neighborhood $U$ 
of $\xi$ such that the path component of $\xi$ in $U$ is homeomorphic to the cone over a Cantor set, 
with $\xi$ corresponding to the vertex of the cone. \\

A path $c: [0,1] \rightarrow \ibd X$ is \textit{safe} if $c(t)$ is a vertex for only finitely
 many $t \in [0, 1]$. Since the property of being join-able by a safe path is an equivalence relation on 
pairs of points, and since $\ibd B_1 \cup \ibd B_2$ is safe path connected when $B_1$ is 
adjacent to $B_2$, it follows that $\bigcup_{B \in \mathscr{B}} \ibd B$ is safe path connected. 
It is shown in \cite{ckpaper} that $\bigcup_{B \in \mathscr{B}} \ibd B$ is
a \textit{safe-path component} of $\ibd X$.

\subsection{The Other Nerve}
In the original paper \cite{ckpaper}, the term "nerve" refers to the tree associated with the block decomposition described in the previous section. There is another Bass-Serre tree of the space that corresponds to the following amalgamated product decomposition:

\begin{figure}[h]
\begin{center}
\begin{tikzpicture}[scale=0.6]

\node (v11) at (0,0) {$T_2$};

\node (v21) at (3,4) {$T_1$};
\node (v22) at (3,3) {$T_1$};
\node (v23) at (3,1) {$T_1$};
\node (v24) at (3,-1) {$T_3$};
\node (v25) at (3,-2) {$T_3$};
\node (v26) at (3,-4) {$T_3$};

\node (v31) at (6,5) {$T_2$};
\node (v32) at (6,4) {$T_2$};
\node (v33) at (6,2) {$T_2$};
\node (v34) at (6,0) {$T_2$};
\node (v35) at (6,-1) {$T_2$};
\node (v36) at (6,-3) {$T_2$};

\draw (v11) to (v21);
\draw (v11) to (v22);
\draw (v11) to (v23);
\draw (v11) to (v24);
\draw (v11) to (v25);
\draw (v11) to (v26);

\draw (v21) to (v31);
\draw (v21) to (v32);
\draw (v21) to (v33);
\draw (v24) to (v34);
\draw (v24) to (v35);
\draw (v24) to (v36);

\node at (3,2) {$\vdots$};
\node at (3,-3) {$\vdots$};
\node at (6,3) {$\vdots$};
\node at (6,-2) {$\vdots$};

\end{tikzpicture}
\end{center}
\caption{The Nerve}
\label{}
\end{figure}
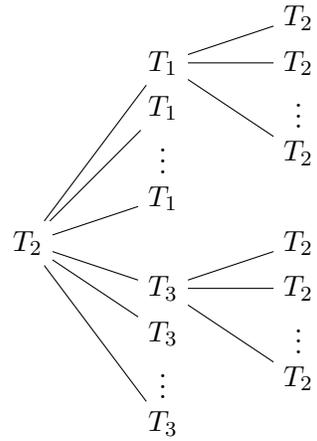

This Bass-Serre tree, which we call the "nerve", records the intersection of special flats and gluing lines rather than blocks and barriers. The \textit{nerve} is a simplicial tree that records the intersections of special flats that are lifts of the tori. Each vertex in the \textit{nerve} represent a special flat that is a lift of $T_1, T_2$ or $T_3$, two vertices are adjacent if and only if two special flats intersect along a line. The nerve is also a regular, infinite and locally infinite tree. Each vertex is adjacent to countably many other vertices. The vertices representing lifts of $T_1$ (and respectively $T_3$) are adjacent to countably infinite copies of vertices representing lifts of $T_2$, while the lifts of $T_2$ are adjacent to both countably many $T_1$ and countably many $T_3$.

\section{Proof of Theorem 1.1}

In this section we restate and proof the main result of this paper, Theorem 1.1.  

Let $W$ be the following right-angled Coxeter group:

$$
W =\{ s_1, s_2,...s_n | s_i^2 \text{ for all }i, [s_i,s_j] \text{ for some pairs } i, j  \}
$$

\noindent The main theorem of this paper is the following:

\begin{thm} 
Suppose $W$ acts geometrically on a Croke-Kleiner space and preserve special flats, then the intersection angle on the 
middle torus must be $\frac{\pi}{2}$.
\end{thm}

Recall a \textit{special flat} is a flat that is a lift of $T_i$. We start with lemmas about the action of the generators and the stabilizers of each special flat.

\begin{lem} \label{withoutinversion}
 For the given complex, each generator of the right-angled Coxeter group acts on the nerve tree without inversion. 
\end{lem}
\begin{proof}
In the nerve, we have vertices that are labeled by the torus of which they denote the universal cover, by abuse of notation, the edges in the nerve are also labeled by either of the following pairs:
$$
\{ T_2, T_1\}, \{T_2, T_3\}
$$
Isometry in the space induces a homeomorphism on the boundary. The boundary of the special flats $\widetilde{T_2}$ is a circle
with four poles, the boundary of the special flats $\widetilde{T_1}$ and $\widetilde{T_3}$ are circles with two poles. Since an isometry of the space induces a homeomorphism on the boundary that takes circles to circles and poles to poles, an isometry cannot invert edges. 
\end{proof}

In the following definitions we lay out O'Brien's construction \cite{mob} of a strict fundamental domain:

\begin{defn}
For a group element $w \in W$, let $X^w$ be the fixed point set of $w$. Let $\mathscr{T}_w (Y)$ be the set of components of $X/ X^w$. If $w= s_i$ is a generator we simply write $\mathscr{T}_i$.

Let
$$
\mathscr{T} : = \bigcup_{i \in I} \mathscr{T}_i
$$ 
\end{defn}

Let $T$ denote a connected  component of $\mathscr{T}_i$, and let 

$$
\tilde{\mathscr{T}}_i := \{ T \cup X^{s_i} | T \in \mathscr{T_i}\}
$$

The elements of $\tilde{\mathscr{T}}_i$ we denote as $\tilde{T} : = T \cup X^{s_i}$.\\

We know that if $s_i$ and $s_j$ does not commute, then $X^{s_i} \cap X^{x_i} = \phi$. Moreover, since $X^{s_j}$ is connected, there exists a unique component in $\mathscr{T}_i$ that contains $X_{s_j}$, which we denote $T^{y}_{ij}$, a \textit{yes}-component, and a \textit{no}-component 
is defined as $T^{n}_{ij} := s_i . T^{y}_{ij}$.

Let 

$$
\tilde{T^{y}} := \bigcap_{ i \in I } \bigcup _{j \in \underline{kl(i)}} (T^{y}_{ij} \cup X^{s_i})
$$

In the context of this paper, suppose a fixed-point set for a generator contains two or more vertices, then it fixes point-wise a unique path in the finite subtree(that is the strict fundamental domain) between them. Since $s$ is order two, the path is forced to extend infinitely which contradicts the assumption that the action is cocompact and therefore the strict fundamental domain is compact. So we have the following corollary: 

\begin{cor} For each generator $s_i$, let $v_{s_i}$ be the label of the fixed point of $s_i$ in the Bass-Serre tree. Then there is exactly a strict fundamental domain where for each $s_i$, $v_{s_i}$ appears as a label for some vertex. 
\end{cor}

\begin{lem}
Suppose a right-angled Coxeter group $W$ acts geometrically on the Croke-Kleiner space and takes special flats to special flats. If a group element $w$ fixes a special flat $\widetilde{T_i}$ set-wise, suppose $w=s_{k}s_{k-1}...s_2s_1$, then each $s_i$ 
fixes $\widetilde{T_i}$ set-wise.
\end{lem}
\begin{proof}

Without loss of generality, suppose $s_1$ does not fix the special flat $T_i$, otherwise let $w = s_k s_{k-1}...s_2$. Let $j$ be the smallest number such that the sub-word $s_js_{j-1}...s_2s_1$ fixes the $T_i$. Consider generators $s_1$ and $s_j$. In $\T_0$,  $s_j$ and $s_1$ each label a vertex, $v_{s_j}$ and $v_{s_1}$. $\T_0$ also contains a lift of $T_i$, label it $v_0$. Since $\T_0$ is a tree, there are unique paths $(v_0, v_{s_j})$ and $(v_0, v_{s_1})$. The word $s_{j-1}s_{j-2}...s_2$ takes the edges $(v_0, v_{s_j})$ to the edges $(v_0, v_{s_1})$. This contradicts the assumption that $T_2$ is a strict fundamental domain. Therefore, each $s_i$ fixes $\widetilde{T_i}$ set-wise.
\end{proof}

Now we can analyze the stabilizer subgroups of each $T_i$.

\begin{prop}
 Given the universal cover of $T_i$, denoted $\widetilde {T_i}$ 
consider the stabilizer subgroup $Stab(\widetilde{T_i})$, then 
$Stab(\widetilde{T_i})$ is generated by a (conjugate) of a subset of the generating set $\{s_1, s_2,...s_n \}$, respectively.
\end{prop}

\begin{proof} Each generator acts simplicially on the nerve tree of special flats. Furthermore, let every edge has length 1, then each group element acts isometrically on the tree. Each generator is of order two. Therefore the fixed point set of each generator acting on the nerve tree is either an induced subgraph or the midpoint of an edge. Lemma \ref{withoutinversion} rules out the latter case. By Corollary 4.4,  there exists a minimal finite tree that is the strict fundamental domain of $W$ on the nerve,
which we denote by $\T$. This tree is the fundamental domain of the group acting on this tree, therefore generators of $W$ does not take points of $\T$ to points of $\T$. By Lemma 4.5, each group element that stabilizes a vertex of $\T_0$ is generated by a subset of generators that stabilizes the vertex. Thus $Stab(\widetilde{T_i})$ are \textbf{special subgroups}, i.e. they are generated by a subset of generators.

\end{proof}

We claim the group $Stab(\widetilde{T_i})$acts on $\widetilde{T_i}$isometrically and
cocompactly. 

\begin{prop}
Stab$(\widetilde{T_i})$acts cocompactly and by isometries on the special flat.   
\end{prop}  

\begin{proof}
The group acts cocompactly on the the space. If $K$ is a fundamental domain for $W \curvearrowright X$, then 
$$
K \cap \widetilde{T_i}
$$
is the fundamental domain for the actions of $Stab(\widetilde{T_i})$. Therefore $Stab(\widetilde{T_i})$ acts cocompactly on the special flat it stabilizes. 
\end{proof}

Next we study a right-angled Coxeter group acting cocompactly and by isometries on a 2-dimensional Euclidean special flat. 
For a presentation of a right-angled Coxeter group

$$
W =\{ s_1, s_2,...s_n | s_i^2 \text{ for all }i, [s_i,s_j] \text{ for some pairs } i, j  \}
$$

Consider the defining graph of the group. First one can rule out the defining graphs on less than or 
equal to three vertices since they either have 0, 2, or infinitely many ends.  Indeed, the number of ends of a group is a quasi-isometry invariant and the plane has one end so a group with 0 or more than one end cannot act geometrically on the plane by the Sˇvarc-Milnor Lemma. 


Recall Gromov's Theorem \cite{thebible}:

\begin{thm}
If a finitely generated group is quasi-isometric to $\Z^n$ then it contains $\Z^n$ 
as a subgroup of finite index.
\end{thm}

\begin{lem}[Key Lemma]
 Suppose $W$ is a right-angled Coxeter group acting cocompactly and by isometries on the special flat $\mathbb{E}^2$.  
Then we claim that $W$ must be the direct product of two copies of the infinite dihedral group.  
\end{lem}
\begin{proof}
We know that the group $W$ has at least four generators. Since $W$ contains $\Z^2$ as a subgroup of finite index, it is not hyperbolic. By \cite{moussong}, if $\Gamma$ is the defining graph of $W$, then in $\Gamma$ there exists induced subgraphs $A,B$ such that $\langle A \rangle$,  $\langle B \rangle$ are infinite and $A *_{join} B$ is a subgraph of $\Gamma$. In particular, there exists two infinite order elements $\gamma'_1 =s_1t_1, \gamma'_2 = s_2t_2$ such that the subgraph on the vertices $s_1, s_2, t_1, t_2$ is a join of two pairs of non-adjacent vertices. The subgraph is a chord-less 4-cycle, where $s_1$ is adjacent to $s_2$ and $t_2$, and $t_1$ is adjacent to $s_2$ and $t_2$.

The actions of $s_1, t_1, s_2, t_2$ are order-2 isometries of the special flat, which are either reflecting across a straight line $l$, 
or rotate around a point $p$ by $\pi$. Two such elements commute in the following cases:

\begin{itemize}
\item $l_1$ and $l_2$ intersecting at right angle
\item $l \cap p \neq \phi$
\end{itemize}

An infinite order action must be a composition of these order-2 isometries as one of the following cases:

\begin{enumerate}
\item $l \cap p = \phi$
\item $l_1 \cap l_2 = \phi$
\item $p_1 \cap p_2 = \phi$
\end{enumerate}
 
In the $\Z^2$ subgroup, there are two elements of infinite order, both generators in one of the three pairs 
of elements commutes with both generators of another one, not necessarily different, of the three pairs.

(1) and (1), impossible since a point cannot simultaneously coincide with another point off the line and be on the line, 
For the same reason, (1) and (3) is also impossible.

(1) and (2), impossible since there is only one straight line that passes perpendicularly through another line
and a point off that line;

(2) and (3), impossible, since one point cannot be on two parallel lines;

(3) and (3), impossible since one point cannot coincides with two points.

Therefore the only possibility is (2) and (2): two pairs of parallel lines intersecting at right angle.
The defining of this group consists of four vertices and four edges connecting up to a four-gon.

To have this group
as a subgroup of finite index, by the Finite Index Lemma \cite{localcon2} we must have in the defining graph a complete graph joined to the chord-less 4-cycle. This is to say the generators not in the chord-less 4-cycle commutes with the four reflections. By the previous argument, there cannot be order-2 symmetries of the special flat that commutes with all four reflections. Therefore, if a right-angled Coxeter group acts geometrically on a special flat, the actions of the group restricted to the special flat is isomorphic to 
$$
G = D_{\infty} \times D_{\infty} = < a, b, c, d | a^2, b^2, c^2, d^2, [a, c], [a, d], [b,c], [b,d]>
$$
\end{proof}

\noindent A direct corollary is the following,

\begin{cor}
If a right-angled Coxeter group acts geometrically on a special flat, then it is isomorphic to $D_{\infty} \times D_{\infty}$. The $D_{\infty} \times D_{\infty}$ acts on the special flat like two pairs reflections cross lines. Each pair consists of two reflections
whose fixed-point sets are parallel axes, and the two pairs of axes intersect at right angle.
\end{cor}

\noindent Next we study how the stabilizer subgroups piece together and determine the gluing angle of the complex.

\begin{thm}
If a right-angled Coxeter group acts geometrically on the Croke-Kleiner space and preserves special flats, then the angles $\theta_1, \theta_2, \theta_3$ must all be $\pi / 2 $.
\end{thm}
\begin{proof}
Consider a special flat that is a lift of either $T_1$ or $T_3$, without loss of generality, let it be $T_1$. Each of these flats is adjacent to countably lifts of $T_2$. The intersections we can label $l_{2,1}$. The $l_{2,1}$s are cosets of $\Z$ and therefore are parallel, bi-infinite geodesic rays. Since any nontrivial action, when restricted to a special flat, is by reflection, there is necessarily an axis of reflection. If an axis of reflection is at an angle $\theta \neq 0, \pi/2$ with the $l_{2,1}$s, then the reflection takes a copy of $l_{2,1}$ to its image, which is not a copy of $l_{2,1}$. Since the boundary of $l_{2,1}$ is a pair of cone points, and isometry of the space induces homeomorphisms on the boundary, we arrive at a contradiction. therefore the axes of reflection is at angle 0 or $\pi / 2$ with the $l_{2,1}$s. By Key Lemma, the reflection axes are two pairs of parallel lines, and one pair is perpendicular to the other. therefore the reflection axes in a special flat labeled by $T_1, T_3$ is either parallel or perpendicular to the $l_{2,1}$s. Furthermore, since each reflection fixes a plane, by Lemma 4.5, each reflection is a (conjugate of a) generator of the group $W$.\\

Now consider the special flats that are lifts of the middle torus $T_2$. In these flats, there are two sets of intersections with neighboring special flats, labeled accordingly $l_{2,1}$ and $l_{2,3}$. All the $l_{2,1}$s are parallel to one another; all the $l_{2,3}$s are parallel to one another. Consider the angle $\theta$ between $l_{2,1}$ and $l_{2,3}$. Suppose $\theta \neq \pi/2$, then the only possibility for a set of four reflection, configured in the way specified in Key Lemma, can take intersections to intersections is to have them reflect across the diagonals of the unit parallelograms in the special flat,  as shown in Figure 6. In Figure 6,  the solid lines are $l_{2,1}$ and $l_{2,3}$, the dashed lines are the axes of reflections.

\begin{figure}[h]
\begin{center}
\includegraphics[scale=0.3]{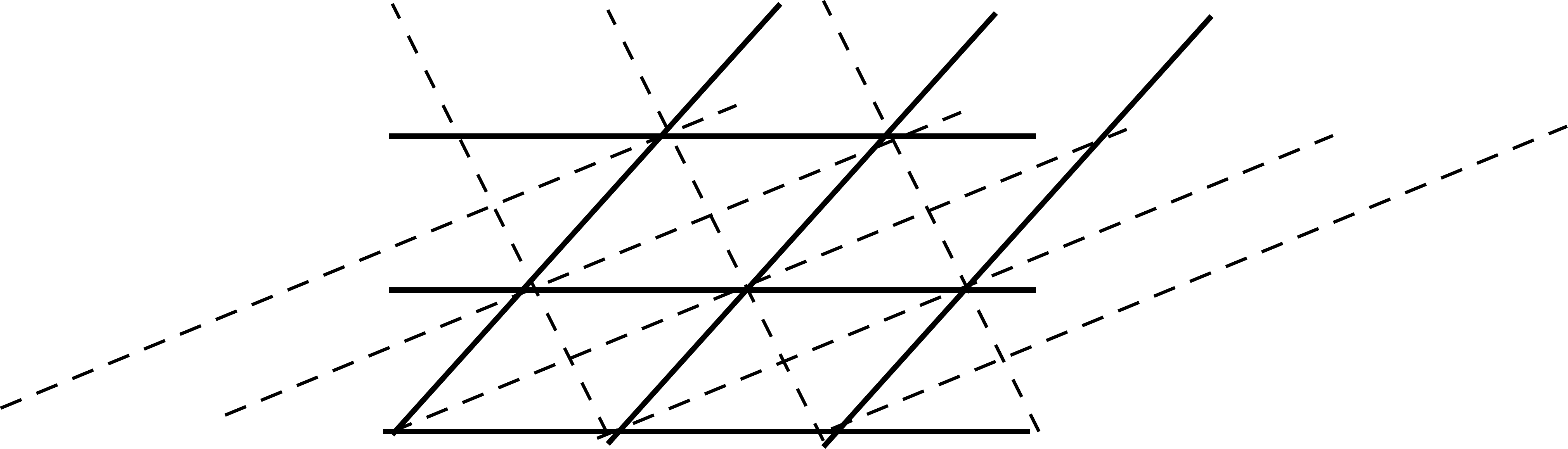}
\end{center}
\caption{Gluing Theorem}
\end{figure}

In this case, it takes a two-letter word to reflect $l_{2,1}$ onto itself across a point.
We argued in the first paragraph that there are generators that reflect $l_{2,1}$ to itself across a point. Since $l_{2,1}$ is also in the flat that is a lift of $T_2$, the same generator then act as reflection on the corresponding $T_2$ and its axis intersects $l_{2,1}$. However there are already reflection axes intersecting $l_{2,1}$ as established in the previous paragraph and neither of them reflect $l_{2,1}$ onto itself. Therefore we need to have a third reflection axis that is not parallel to the two existing axes. This configuration contradicts the Key Lemma. Therefore, it is not possible to have the intersection angle of $l_{2,1}$ and $l_{2,3}$ be $\theta \neq \pi / 2 $.

\end{proof}

\head{Remark.}
The main theorem states that if a right-angled Coxeter group acts geometrically on a Croke-Kleiner space and preserves special flats, then the angle must be fixed at $\frac{\pi}{2}$, therefore, the only parameters for the group action are the length data. We know that changing the length data changes the $G$-equivariant homeomorphism type of the boundary \cite{yulan}, it suffices to verify that changing the length data does not violates the requirement of a isometric, properly discontinuous, and cocompact action.

Since we can vary the distances between two parallel reflecting axes on each special flat, we can indeed obtain actions of varied "translation lengths" and the action is still geometric, therefore, we can conclude from this and \cite{yulan} the following result:

\begin{prop}
If a right-angled Coxeter group acts geometrically on the Croke-Kleiner space $\widetilde{X}$, then changing 
the lengths data of the tori changes the equivariant type of it visual boundary. There are uncountably many equivariant visual boundaries of of the space.
\end{prop}

\section{A Concrete Action}

As discussed in the introduction, we assumed without verifying that there does exist a right-angled Coxeter group that acts geometrically on the Croke-Kleiner space. In this section a specific right-angled Coxeter group is given. In general there may be more than one right-angled Coxeter groups that is quasi-isometric to the Croke-Kleiner space. Let $W$ be the right-angled Coxeter group defined by the following graph:

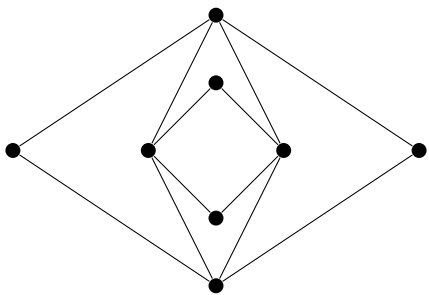
\begin{figure}[h]
\begin{center}
\begin{tikzpicture}[scale=0.9]

\node (v1) [circle,fill,inner sep=2pt] at (-1,0) {};
\node (v2) [circle,fill,inner sep=2pt] at (0,1) {};
\node (v3) [circle,fill,inner sep=2pt] at (1,0) {};
\node (v4) [circle,fill,inner sep=2pt] at (0,-1) {};

\draw (v1) -- (v2) -- (v3) -- (v4) -- (v1);

\node (w1) [circle,fill,inner sep=2pt] at (-3,0) {};
\node (w2) [circle,fill,inner sep=2pt] at (0,2) {};
\node (w3) [circle,fill,inner sep=2pt] at (3,0) {};
\node (w4) [circle,fill,inner sep=2pt] at (0,-2) {};

\draw (w1) -- (w2) -- (w3) -- (w4) -- (w1);

\draw (v1) -- (w2) -- (v3) -- (w4) -- (v1);

\end{tikzpicture}
\end{center}
\caption{Defining graph of a right-angled Coxeter group}
\label{}
\end{figure}

Consider the Cayley graph of $W$ with respect to this generating set. There are three "diamonds". Each "diamond" in the defining graph corresponds to an  $ H = D_{\infty} \times D_{\infty}$ whose Caley graph is a $\Z \times \Z $ lattice with each edge being replaced by a double-edge, i.e. a pair of edges that shares starting and ending vertices. Three diamonds generates three types of such double-edged lattices. These lattices are identified along double-edged $\Z-$lines according to the amalgamated product decomposition:
$$  
W = H *_{D_{\infty}} H *_{D_{\infty}}H 
$$

We observe that the Caley graph embeds into the Croke-Kleiner space as its dual: each vertex of the Cayley graph represent a unit square in the Croke Kleiner space and two vertices of the Cayley graph are connected by a pair of double-edges if and only if the corresponding two unit squares share a common edge in the Croke-Kleiner space. We observe that group $W$ acts in its own Cayley graph by reflecting through the mid-point of the double edges, and it is easy to check that the induced action on the Croke-Kleiner space in which it embeds is indeed a geometric action.\\

\bibliographystyle{plain}
\bibliography{bib1}

\addcontentsline{toc}{chapter}{Bibliography} 

\end{document}